\documentclass[11pt]{article}
\usepackage{amsmath,amsthm,amssymb,amsfonts}
\usepackage{enumerate}
\usepackage{mathrsfs}
\usepackage[cmtip,all]{xy}
\usepackage{color}
\normalsize

\newtheoremstyle{theoremstyle}
  {10pt}      %  Space above
  {5pt}       %  Space below
  {\itshape}  %  Body font
  {}          %  Indent amount (empty = no indent, \parindent = para indent)
  {\bfseries} %  Thm head font
  {:}         %  Punctuation after thm head
  {.5em}      %  Space after thm head: " " = normal interword space;
  {}          %  Thm head spec (can be left empty, meaning `normal')

\newtheoremstyle{examplestyle}
  {10pt}      %  Space above
  {5pt}       %  Space below
  {}          %  Body font
  {}          %  Indent amount (empty = no indent, \parindent = para indent)
  {\bfseries} %  Thm head font
  {:}         %  Punctuation after thm head
  {.5em}      %  Space after thm head: " " = normal interword space;
  {}          %  Thm head spec (can be left empty, meaning `normal')

\theoremstyle{plain}
\newtheorem{theorem}{Theorem}[section]
\newtheorem*{theorem*}{Theorem}
\newtheorem{lemma}[theorem]{Lemma}
\newtheorem{proposition}[theorem]{Proposition}
\newtheorem*{proposition*}{Proposition}
\newtheorem{corollary}[theorem]{Corollary}
\newtheorem*{corollary*}{Corollary}

\theoremstyle{definition}
\newtheorem{definition}[theorem]{Definition}
\newtheorem{definition*}{Definition}
\theoremstyle{remark}
\newtheorem{remark}[theorem]{Remark}
\newtheorem{remark*}{Remark}

\newcommand{\caC}{{\mathcal C}}
\newcommand{\caD}{{\mathcal D}}

\newcommand{\caK}{{\mathcal K}}

\newcommand{\caM}{{\mathcal M}}
\newcommand{\caS}{{\mathcal S}}

\newcommand{\caO}{{\mathcal O}}

\newcommand{\MZ}{\mathsf{M}\mathbf{Z}}

\newcommand{\unit}{\mathbf{1}}

\newcommand{\N}{\mathbb{N}}

\newcommand{\Ho}{\mathbf{Ho}}

\newcommand{\Mod}{\mathrm{Mod}}
\newcommand{\Env}{\mathrm{Env}}

\newcommand{\bL}{\mathbf{L}}
\newcommand{\bR}{\mathbf{R}}

\newcommand{\Hom}{\mathrm{Hom}}
\newcommand{\map}{\mathrm{map}}

\newcommand{\hocolim}{\mathrm{hocolim}}

\newcommand{\Coll}{\mathsf{Coll}}
\newcommand{\Oper}{\mathsf{Oper}}
\newcommand{\Pairs}{\mathsf{Pairs}}
\newcommand{\cOper}{\mathscr{O}}
\newcommand{\cPairs}{\mathscr{P}}
\newcommand{\cA}{\mathscr{A}}
\newcommand{\cB}{\mathscr{B}}
\newcommand{\cM}{\mathscr{M}}
\newcommand{\cR}{\mathscr{R}}

\newcommand{\sSets}{s\mathsf{Sets}}
\newcommand{\Alg}{\mathsf{Alg}}
\newcommand{\Ass}{\mathit{\mathcal{A}ss}}
\newcommand{\Com}{\mathit{\mathcal{C}om}}
\newcommand{\modu}{\mathit{\mathcal{M}od}}

\newcommand{\Map}{\mathrm{Map}}
\newcommand{\End}{\mathrm{End}}

\date{}
\title{{\bf On functorial (co)localization of \\ algebras and modules over operads}}
\author{Javier J. Guti$\acute{\mathrm{e}}$rrez, Oliver R\"ondigs, Markus Spitzweck, Paul Arne {\O}stv{\ae}r}

\begin{document}
\maketitle
\begin{abstract}
Motivated by calculations of motivic homotopy groups, 
we give widely attained conditions under which operadic algebras and modules thereof are preserved under (co)localization functors. 
\end{abstract}
{\footnotesize{\tableofcontents}}
\newpage

\section{Introduction}
\label{section:introduction}

Operads are key mathematical devices for organizing hierarchies of higher homotopies in a variety of settings.
The earliest applications were concerned with iterated topological loop spaces.
More recent developments have involved derived categories, factorization homology, knot theory, moduli spaces, representation theory, string theory, deformation quantization, and many other topics.
This paper is a sequel to our work on operads in the context of the slice filtration in motivic homotopy theory \cite{GRSO}.

The problem we address here is that of preservation of algebras over colored operads, 
and also modules over such algebras, 
under Bousfield (co)localization functors.
For this  we only require a few widely attained technical assumptions and notions on the underlying model categories and the operads,
e.g., that of strongly admissible operads in a cofibrantly generated symmetric monoidal model category.
We refer to \cite{CaRaTo}, \cite{WY2}, and \cite{WY1} for related results on (co)localization of monadic algebras.

Our main motivation for studying the mentioned problem of preservation of algebras is rooted in Morel's $\pi_{1}$-conjecture \cite{oo}, \cite{roendigs-spitzweck-oestvaer.slices-sphere}. 
For a field $F$, 
this conjecture states there exists a short exact sequence of Nisnevich sheaves on the category of smooth $F$-schemes of finite type  
\begin{equation}
\label{eq:first-stable-stem} 
0 
\longrightarrow 
\mathbf{K}^{\mathsf{M}}_{2}/24 
\longrightarrow
\mathbf{\pi}_{1,0}\unit
\longrightarrow 
\mathbf{\pi}_{1,0}\mathbf{KQ}
\longrightarrow 
0.
\end{equation}
Here, $\unit$ is the motivic sphere spectrum, 
$\mathbf{K}^{\mathsf{M}}$ denotes Milnor $K$-theory, 
and $\mathbf{KQ}$ is the hermitian $K$-theory spectrum.  
The solution of Morel's $\pi_{1}$-conjecture \cite{roendigs-spitzweck-oestvaer.slices-sphere} involves an explicit calculation in the slice spectral sequence of the motivic sphere spectrum.  
One of the precursors for this calculation is the fact that the total slice functor takes $E_{\infty}$ motivic spectra, 
in particular the algebraic cobordism spectrum, 
to graded $E_{\infty}$ $\MZ$-algebras in a functorial way.  
Here, $\MZ$ denotes the motivic Eilenberg-MacLane spectrum.  
Theorems \ref{theorem:colocalcommute} and \ref{theorem:lkjoij} in this paper coupled with our construction of the slice filtration in \cite[\S6]{GRSO} 
verify the mentioned multiplicative property (which in turn is used in the proof of \cite[Theorem 2.20]{roendigs-spitzweck-oestvaer.slices-sphere}).  
We envision that future calculations with slice spectral sequences will exploit multiplicative structures to a greater extent, 
and as such will be relying on the results herein.

The paper starts with \S\ref{section:Modelstructuresofoperadsandalgebras} on model structures on operads and algebras.
Our main results on preservation of algebras and modules under Bousfield (co)localization functors are shown in \S\ref{(Co)localizationofalgebrasovercoloredoperads} and
\S\ref{section:proofsofthemainresults}.
To make the paper reasonably self-contained we have included two appendices fixing our conventions on model categories and colored operads.
In particular, we review tensor-closed sets of objects in a homotopy category, the Reedy model structure, operadic algebras, and modules over such algebras.

\section{Model structures of operads and algebras}
\label{section:Modelstructuresofoperadsandalgebras}
Let $\caC$ be a cocomplete closed symmetric monoidal category with tensor product $\otimes$, unit $I$, initial object $0$, and internal hom functor $\Hom(-,-)$.
For a set $C$ we refer to Appendix~\ref{jnhgd} for the definitions of $C$-colored collections and $C$-colored operads in~$\caC$.
Recall that a $C$-colored collection $\mathcal{K}$ is \emph{pointed} if it is equipped with \emph{unit} maps $I\to \mathcal{K}(c;c)$ for every $c\in C$.  
Denote by $\Coll_C(\caC)$ and $\Coll_C^{\bullet}(\caC)$ the categories of $C$-colored collections and pointed $C$\nobreakdash-colored collections, respectively. 
If $\mathcal{K}$ is a $C$-colored collection, we can define a pointed $C$-colored collection $F(\mathcal{K})$ by setting
$F(\mathcal{K})(c;c):=\mathcal{K}(c;c)\coprod I$ for every $c$ in $C$, and $F(\mathcal{K})(c_1,\ldots, c_n;c):=\mathcal{K}(c_1,\ldots, c_n; c)$ if $n\ne 1$. 
This defines the free-forgetful adjoint functor pair
$$
\xymatrix{
F\colon \Coll_C(\caC) \ar@<0.50ex>[r] & \ar@<0.50ex>[l] \Coll_C^{\bullet}(\caC)\colon U.
}
$$
We denote by $\Oper_C(\caC)$ and $\Oper(\caC)$ the categories of $C$-colored operads and (one-colored) operads in~$\caC$, respectively.

Suppose $\caC$ is a cofibrantly generated symmetric monoidal model category. 
Then $\Coll_C(\caC)$ and $\Coll_C^{\bullet}(\caC)$ have transferred model structures, where weak equivalences and fibrations are defined colorwise.
There is a free-forgetful adjoint pair
\begin{equation}
\label{free-forgetful}
\xymatrix{
F\colon \Coll_C^{\bullet}(\caC)\ar@<0.50ex>[r] & \ar@<0.50ex>[l] \Oper_C(\caC)\colon U.
}
\end{equation}
Under suitable conditions, 
the model structure on (pointed) $C$-colored collections can be transferred along \eqref{free-forgetful} to a cofibrantly generated model structure on $\Oper_C(\caC)$, 
in which a map of $C$\nobreakdash-colored operads is a fibration or a weak equivalence if its underlying (pointed) $C$\nobreakdash-colored collection is so. 
This holds for $k$-spaces, simplicial sets, and symmetric spectra;
see~\cite[Theorems 3.1, 3.2]{BM03}, \cite[Theorem 2.1, Example 1.5.6]{BM07} and \cite[Corollary 4.1]{GV12}.

In general, 
\eqref{free-forgetful} does not furnish a model structure on $\Oper_C(\caC)$, 
but rather the weaker structure of a \emph{semi model structure}.
In a semi model category the axioms of a model category hold with the exceptions of the lifting and factorization axioms, 
which hold only for maps with cofibrant domains. 
The trivial fibrations have the right lifting property with respect to cofibrant objects, since the initial object of a semi model category is assumed to be cofibrant.
For operads the following result is shown in~\cite[Theorem 3.2]{Spi} (cf.~\cite[Theorem 12.2A]{Fre09}).
Our extension to colored operads follows similarly.

\begin{theorem}
If $\caC$ is a cofibrantly generated symmetric monoidal model category, 
then the model structure on $\Coll^{\bullet}_C(\caC)$ transfers along the free-forgetful adjunction \eqref{free-forgetful} to a cofibrantly generated semi model structure on $\Oper_C(\caC)$, 
in which a map $\mathcal{O}\rightarrow\mathcal{O}'$ is a fibration or a weak equivalence if $\mathcal{O}(c_1,\ldots,c_n;c)\rightarrow \mathcal{O}'(c_1,\ldots, c_n; c)$ 
is a fibration or a weak equivalence in $\caC$, respectively, for every tuple of colors $(c_1,\ldots, c_n, c)$.
\end{theorem}

Throughout the paper we will implicitly assume that $\Oper_C(\caC)$ \emph{always} admits a cofibrantly generated transferred model structure,
where the weak equivalences and fibrations are defined at the level of the underlying collections. 

Let $\caC^C$ denote the product category $\prod_{c\in C}\caC$. 
If $\mathcal{O}$ is a $C$-colored operad, denote by $\Alg_{\mathcal{O}}(\caC)$ the category of $\mathcal{O}$-algebras in $\caC$; 
see Appendix~\ref{jnhgd}. There is a free-forgetful adjoint pair
\begin{equation}
\label{free-forget}
\xymatrix{
F_{\mathcal{O}}\colon \caC^C \ar@<0.50ex>[r] & \ar@<0.50ex>[l] \Alg_{\mathcal{O}}(\caC)\colon U_{\mathcal{O}},
}
\end{equation}
where the left adjoint is the free $\mathcal{O}$-algebra functor defined by 
$$
F_{\mathcal{O}}(\cA)(c)=\coprod_{n\ge 0}\left(\coprod_{c_1,\dots, c_n\in C}\mathcal{O}(c_1,\ldots, c_n;c)
\otimes_{\Sigma_n} \cA(c_1)\otimes\cdots\otimes \cA(c_n) \right).
$$
If it is clear from the context we shall write $F$ and $U$ instead of $F_\caO$ and $U_\caO$, respectively.

Let $\caC$ be a cofibrantly generated symmetric monoidal model category. 
Recall from \cite{BM03} that a $C$-colored operad $\mathcal{O}$ is \emph{admissible} if the product model structure on $\caC^C$ transfers to a cofibrantly generated model structure on 
$\Alg_{\mathcal{O}}(\caC)$ via \eqref{free-forget}.
An $\caO$-algebra $\cA$ is \emph{underlying cofibrant} if $U(\cA)$ is cofibrant in $\caC^C$; 
i.e., $\cA(c)$ is cofibrant in $\caC$ for all $c\in C$.

As indicated in~\cite[I.5]{Spi}, 
if $\caC$ is a simplicial symmetric monoidal model category and $\mathcal{O}$ is an admissible $C$-colored operad, 
then $\Alg_{\caO}(\caC)$ is naturally a simplicial model category. 
For a simplicial set $K$ and an $\caO$-algebra $\cA$, the cotensor $\cA^K$ is the object $(U_{\caO} \cA)^K$ with $\caO$-algebra structure given by the composition 
$\caO\to \End(\cA)\to\End(\cA^K)$ 
--- for the endomorphism colored operad --- 
induced by the diagonal map $K\to K\times\cdots\times K$. 
For $K$ fixed, the functor $(-)^K$ has a left adjoint defining the tensor.
For $\cA$ fixed, the functor $\cA^{(-)}$ has a right adjoint defining the simplicial enrichment in $\Alg_{\caO}(\caC)$.

\begin{definition}
\label{def:strongly_admissible}
Let $\caC$ be a cofibrantly generated symmetric monoidal model category. 
A $C$-colored operad $\caO$ in $\caC$ is \emph{strongly admissible} if there is a weak equivalence $\varphi\colon \caO' \to \caO$ of admissible $C$-colored operads 
inducing a Quillen equivalence
$$
\xymatrix{
\varphi_!\colon \Alg_{\caO'}(\caC) \ar@<0.50ex>[r] & \ar@<0.50ex>[l] \Alg_{\caO}(\caC)\colon \varphi^*
}
$$
and $\caO'$ satisfies one of the conditions:
\begin{itemize}
\item[{\rm (i)}] It has an underlying cofibrant $C$-colored collection.
\item[{\rm (ii)}] It has an underlying cofibrant pointed $C$-colored collection, 
and $\caC$ has an additional cofibrantly generated symmetric monoidal model structure with the same weak equivalences and cofibrant unit.
\end{itemize}
We call the triple $(\caO, \caO', \varphi)$ a \emph{strongly admissible pair}.
\label{def:strongly_admissible_pair}
\end{definition}

\begin{remark}
By \cite[Theorem 1]{Mur15} any combinatorial symmetric monoidal model category satisfying the very strong unit axiom admits a combinatorial symmetric monoidal 
model structure with the same weak equivalence and (possibly) more cofibrations making the unit cofibrant.
The \emph{very strong unit axiom} says that tensoring any object with a cofibrant approximation of the unit is a weak equivalence. 
This holds in many examples, 
e.g., when tensoring with cofibrant objects preserve weak equivalences~\cite[Corollary~9]{Mur15}.
\end{remark}

\begin{remark}
\label{rem:cofibrant_unit}
If $\caC$ is a simplicial symmetric monoidal model category, then the unit of $\caC$ is cofibrant:
For any monoidal Quillen adjunction $i\colon\sSets\rightleftarrows\caC\colon r$, $i$ preserves the unit and cofibrations.  
Thus if $\caO$ is a $C$-colored operad in $\caC$ with an underlying cofibrant pointed $C$-colored collection, 
$\caO$ has an underlying cofibrant $C$-colored collection.
\end{remark}

Let $\cA$ be a monoid in a closed symmetric monoidal category $\caC$. 
Define the operad $\caO_\cA$ by $\caO_\cA(n)=\cA$ if $n=1$ and zero otherwise. 
The algebras over $\caO_\cA$ in $\caC$ are precisely the $\cA$-modules. 
A map of monoids $\cA\to \cB$ induces a map of operads $\caO_\cA\to \caO_{\cB}$.

\begin{definition}
Let $\caC$ be a cofibrantly generated symmetric monoidal model category. 
A monoid $\cA$ in $\caC$ is \emph{strongly admissible} if there is another monoid $\cA'$ and a weak equivalence $\varphi\colon \cA'\to \cA$ such that $(\caO_\cA, \caO_{\cA'},\varphi)$ 
is a strongly admissible pair.
\end{definition}

The constant simplicial object functor sends an object $X$ to the simplicial object $X_{\bullet}$ with $X_n=X$ for all $n$. 
If $\caC$ is symmetric monoidal, this is a symmetric monoidal functor for the objectwise tensor product on $s\caC$. 
Thus, if $\caO$ is a $C$-colored operad in $\caC$, we can view it as a $C$-colored operad in the category of simplicial objects $s\caC$ by applying the constant functor levelwise.

\begin{lemma} 
\label{corollary:bgfrd}
Suppose $\caO$ is an admissible $C$\nobreakdash-colored operad in a simplicial symmetric monoidal model category $\caC$.
For every simplicial object $\cA_\bullet$ in $\Alg_{\caO}(\caC)$ there is a natural isomorphism
$$
|U(\cA_\bullet)|_{\caC^C} \cong U(|\cA_\bullet|_{\Alg_{\caO}(\caC)}),
$$
where $U$ and $|-|$ denote the corresponding forgetful and realization functor, respectively. 
\end{lemma}
\begin{proof}
In any simplicial model category there are adjoint functors $|-|_\caC\colon s \caC \to \caC $ and $\mathrm{Sing}_{\caC}\colon \caC\to s\caC$,
where $\mathrm{Sing_{\caC}}(X)$ is the simplicial object with $\mathrm{Sing_{\caC}}(X)_n=X^{\Delta[n]}$. Since $\Alg_{\caO}(\caC)$ is also a simplicial model category, 
we have the adjunction
$$
\xymatrix{
|-|_{\Alg_{\caO}(\caC)} \colon s \Alg_{\caO}(\caC) \ar@<0.50ex>[r] & \ar@<0.50ex>[l]  \Alg_{\caO}(\caC) \colon \mathrm{Sing}_{\Alg_{\caO}(\caC)}.
}
$$
By Lemma \ref{lemma:fgngd} (ii) $| - |_\caC$ is a symmetric monoidal functor.
Hence there exists an induced adjunction between $\caO$-algebras in $s\caC$ and $\caC$. 
But the $\caO$-algebras in $s \caC$ 
--- viewing $\caO$ as a constant simplicial object in $s\caC$ --- are precisely $s\Alg_{\caO}(\caC)$.  
We claim the two adjoint pairs are isomorphic. 
Indeed, 
if $\cA$ is an $\caO$-algebra, 
$\cA^{\Delta[n]}$ is $(U_{\caO}\cA)^{\Delta[n]}$ with $\caO$-algebra structure given by the composite $\End(\cA)\to \End(\cA^{\Delta[n]})$, 
as explained in~\cite[I.5]{Spi}. 
It follows easily that the right adjoints coincide.
\end{proof}

\begin{proposition} 
\label{proposition:rcn}
Suppose $\caO$ is an admissible $C$\nobreakdash-colored operad in a cofibrantly generated symmetric monoidal model category $\caC$.
\begin{itemize}
\item[{\rm (i)}] If $\caO$ has an underlying cofibrant $C$-colored collection, then every cofibrant $\caO$\nobreakdash-algebra is underlying cofibrant.
\item[{\rm (ii)}] If $\caO$ has an underlying cofibrant pointed $C$-colored collection and $\caC$ has a second symmetric monoidal model structure 
with the same weak equivalences and cofibrant unit, then every cofibrant $\caO$-algebra is underlying cofibrant in this model structure.
\end{itemize}
\end{proposition}
\begin{proof}
The proof for operads in \cite[Corollary 5.5]{BM03} extends to colored operads. (The proof of \cite[Theorem 4.1]{BM07} gives closely related steps.) 
Alternatively, use the colored operads version of \cite[Proposition 4.8]{Spi}.
\end{proof}

We refer to Appendix \ref{hegvfkehvck} for the Reedy model structure on simplicial categories.

\begin{lemma} 
\label{lemma:bgfhj}
Suppose $\caO$ is a $C$\nobreakdash-colored operad with an underlying cofibrant collection in a cofibrantly generated symmetric monoidal model category $\caC$.
Then $\caO$ --- viewed as an operad in $s\caC$ via the constant functor --- has an underlying cofibrant $C$-colored collection in $s\caC$.
\end{lemma}
\begin{proof}
Suppose $G$ is a discrete group and $\caC^G$ is the category of objects in $\caC$ with right $G$-actions.
Then the Reedy model structure on $s(\caC^G)$ 
--- for the transferred model structure on $\caC^G$ --- 
coincides with the model structure on $(s\caC)^G$ transferred from the Reedy model structure on $s\caC$. 
Thus the corresponding model structures on $s\Coll_C(\caC)$ and $\Coll_C(s\caC)$ coincide.
Recall that  $\Delta^{\rm op}$ has cofibrant constants \cite[Co\-ro\-llary 15.10.5]{Hirschhorn}.
Thus cofibrancy of the underlying $C$\nobreakdash-colored collection of $\caO$ in $\Coll_C(\caC)$ implies
the underlying $C$-colored collection of $\caO$ --- viewed as a constant simplicial object --- is Reedy cofibrant in $s\Coll_C(\caC)$ \cite[Theorem 15.10.8(1)]{Hirschhorn}, 
and hence it is cofibrant in $\Coll_C(s\caC)$.
\end{proof}

\begin{lemma} 
\label{lemma:gfedg}
Suppose $\caO$ is an admissible $C$\nobreakdash-colored operad in a symmetric monoidal model category $\caC$. 
Then $s \Alg_{\caO}(\caC)$ has a model structure transferred from $(s \caC)^C$ --- equipped with the colorwise Reedy model structure --- which coincides with its Reedy model structure.
\end{lemma}
\begin{proof}
We show that the Reedy model structure on $s \Alg_{\caO}(\caC)$ is the transferred model structure from $(s \caC)^C$. 
Since the weak equivalences are defined objectwise in the Reedy model structure, 
the weak equivalences in $s\Alg_{\caO}(\caC)$ are precisely the maps that become weak equivalences in $(s\caC)^C$. 
A map in $s \Alg_{\caO}(\caC)$ is a Reedy fibration if certain maps involving matching objects and fiber products are fibrations in $\Alg_{\caO}(\caC)$.
But since the fibrations in $\Alg_{\caO}(\caC)$ are the underlying fibrations and the matching objects and fiber products commute with taking the underlying collection, 
the result follows.
\end{proof}

\begin{corollary} 
\label{corollary:gfddf}
Suppose $\caO$ is an admissible $C$-colored operad in a cofibrantly generated symmetric monoidal model category $\caC$.
\begin{itemize}
\item[{\rm (i)}] If $\caO$ has an underlying cofibrant $C$-colored collection, then any Reedy cofibrant object in $s \Alg_{\caO}(\caC)$ is Reedy cofibrant as an object in $s \caC$.
\item[{\rm (ii)}] Suppose $\caO$ has an underlying cofibrant pointed $C$-colored collection and $\caC$ has a second symmetric monoidal model structure with the same weak equivalences 
and cofibrant unit.  Then any Reedy cofibrant object in $s \Alg_{\caO}(\caC)$ is cofibrant in $s \caC$ equipped with the Reedy model structure induced by this model structure on $\caC$.
\end{itemize}
\end{corollary}
\begin{proof}
The category of $\caO$-algebras in $s \caC$ has a model structure transferred from $(s \caC)^C$ by assumption and Lemma~\ref{lemma:gfedg}. 
Moreover, $s \caC$ is a symmetric monoidal model category by Lemma~\ref{lemma:dfgfg} (i).
Also every object in $s\caC$ is small relative to the whole category.

To prove part (i) note that the constant operad on $\caO$ in $s\caC$ has an underlying cofibrant collection by Lemma~\ref{lemma:bgfhj}. 
Thus we can apply Proposition~\ref{proposition:rcn}(i) to $s \caC$ with the Reedy model structure. 
Since a Reedy cofibrant object of $s \Alg(\caO)$ is cofibrant for the transferred model structure, 
by Lemma \ref{lemma:gfedg}, 
this gives the result.

Part (ii) is proved similarly by reference to \cite[Proposition 4.8]{Spi}.
(By assumption the constant operad on $\caO$ in $s\caC$ has an underlying cofibrant collection in $s\caC$ for the Reedy model structure induced by the second model structure on $\caC$.) 
\end{proof}

\section{(Co)localization of algebras over colored operads}
\label{(Co)localizationofalgebrasovercoloredoperads}

\subsection{Colocalization of algebras}
In this section we show that tensor-closed $\caK$\nobreakdash-colocalization functors preserve algebras over cofibrant $C$-colored operads. 
More precisely, we prove that if $\caO$ is a cofibrant $C$-colored operad and $f$ is an $\bL F(\caK)$-colocalization in the category of $\caO$-algebras $\Alg_{\caO}(\caC)$, 
then $U(f)$ is a $\caK$\nobreakdash-colocalization in $\caC$, where $U$ denotes the forgetful functor.

If $\caK$ is a set of isomorphism classes of objects of $\Ho(\caC)$, and $C$ is a set of colors, denote by $\caK^C$ the set of objects in $\Ho(\caC)^C$ defined as $\caK^C=\prod_{c\in C}\caK$.
Note that an object in $\caC^C$ is $\caK^C$-colocal if and only if it is colorwise $\caK$-colocal.
\begin{lemma} 
\label{lemma:hocolimUweakequivalence}
Suppose $\caO$ is a strongly admissible $C$-colored colored operad in a cofibrantly generated simplicial symmetric monoidal model category $\caC$.
For a simplicial object $\cA_\bullet$ in $\Alg_{\caO}(\caC)$, the canonical map
$$
\hocolim_{\Delta^\mathrm{op}} U(\cA_\bullet) \longrightarrow U(\hocolim_{\Delta^\mathrm{op}} \cA_\bullet)
$$
is a weak equivalence, where $U$ denotes the corresponding forgetful functor.
\end{lemma}
\begin{proof}
Since $\caC$ is a simplicial symmetric monoidal model category, Remark~\ref{rem:cofibrant_unit} shows we may assume $\caO$ is strongly monoidal and satisfies 
Definition~\ref{def:strongly_admissible}(i). 
We may also assume that $\caO$ has an underlying cofibrant collection, due to the Quillen equivalence between $\Alg_{\caO}(\caC)$ and $\Alg_{\caO'}(\caC)$ in the definition of a strongly 
admissible operad. 

By homotopy invariance of the homotopy colimit, we may further assume that $\cA_\bullet$ is Reedy cofibrant. 
By Lemma~\ref{lemma:nbgfgfb}, 
$s \caC$ is cofibrantly generated.  
Thus by Corollary \ref{corollary:gfddf}(i) $U(\cA_\bullet)$ is Reedy cofibrant as well.
By Lemma \ref{lemma:ngfdv}, 
$|\cA_\bullet|_{\Alg(\caO)}$ computes the homotopy colimit of $\cA_\bullet$, and $|U(\cA_\bullet)|_{\caC^C}$ computes the homotopy colimit of $U(\cA_\bullet)$. 
Corollary \ref{corollary:bgfrd} gives an isomorphism $|U(\cA_\bullet)|_{\caC^C} \cong U(|\cA_\bullet|_{\Alg(\caO)})$,
which finishes the proof.
\end{proof}

Let $\caC$ be a symmetric monoidal model category and $\caO$ a $C$-colored operad in $\caC$. 
Given any $\caO$\nobreakdash-algebra $\cA$ in $\Alg_{\caO}(\caC)$ we define the \emph{standard simplicial object associated to $\cA$} by setting $\cA_n=(FU)^{n+1} \cA$ with the usual 
structure maps. 
Here, 
$F$ and $U$ denote the free functor and the forgetful functor, respectively. 
There is a canonical augmentation $\cA_{\bullet}\to \cA$ obtained by viewing $\cA$ as a constant simplicial object.

\begin{lemma} 
\label{lemma:hocolimmap}
Suppose $\caO$ is a strongly admissible colored operad in a cofibrantly generated simplicial symmetric monoidal model category $\caC$. 
For every $\caO$-algebra $\cA$, the augmentation map induces a canonical weak equivalence $\hocolim_{\Delta^{\mathrm{op}}}\cA_\bullet \to \cA$.
\end{lemma}
\begin{proof}
This follows from Lemma~\ref{lemma:hocolimUweakequivalence} since $U(\cA_\bullet) \to U(\cA)$ is a split augmented simplicial object.
\end{proof}

\begin{remark}
In Lemma \ref{lemma:hocolimmap} one should be mindful of forming the ``correct" derived simplicial object, 
i.e., 
in degree $n$ it is weakly equivalent to $(FQU)^{n+1} \cA$, 
where $Q$ is a cofibrant replacement functor in $\caC^C$.
\end{remark}

\begin{lemma}
\label{lemma:underlyingcolocal} 
Let $\caO$ be a strongly admissible $C$-colored operad in a cofibrantly generated simplicial symmetric monoidal model category $\caC$, 
and $\caK$ a tensor-closed set of isomorphism classes of objects of $\Ho(\caC)$. 
Suppose $\caO(c_1, \ldots, c_n;c)\otimes^{\bL} -$ preserves $\caK$-colocal objects for all $(c_1,\ldots, c_n,c)$, $n\ge 0$.
If $X$ in $\caC^C$ is colorwise $\caK$-colocal, 
then $\bL F(X)$ is underlying colorwise $\caK$-colocal.
\end{lemma}
\begin{proof}
Since $\caO$ is strongly admissible and $\caC$ is a simplicial symmetric monoidal model category,
we may assume that $\caO$ has an underlying cofibrant collection.  
We note that $\caO(c;c) \otimes X$ is cofibrant in $\caC$ for every $c \in C$ and cofibrant $X$ in $\caC$.
Moreover, we have  
$$
F(X)(c)=\coprod_{n \ge 0} \left( \coprod_{d \in C^n}
\caO(d_1,\ldots,d_n;c) \otimes_{\Sigma_n} X(d_1) \otimes \cdots \otimes X(d_n) \right).
$$
The result follows now from the fact that $\caK$ is tensor-closed,
$\caK$-colocal objects are closed under coproducts, 
and $F(X)(c)$ is a homotopy quotient of $\caK$-colocal objects for every $c\in C$, 
hence $\caK$-colocal.
\end{proof}
\begin{remark}
If $\caO(c_1,\ldots, c_n;c)$ is $\caK$-colocal, 
then $\caO(c_1,\ldots, c_n;c)\otimes -$ preserves $\caK$-colocal objects for all $(c_1,\ldots, c_n, c)$, $n\ge 0$, 
since $\caK$ is tensor-closed. 
The converse holds provided the unit $I$ is $\caK$-colocal.
\end{remark}

\begin{lemma} 
\label{lemma:hocolimunderlying}
Under the same assumptions as in Lemma \ref{lemma:underlyingcolocal},
let $D\colon\mathcal{I} \to \Alg_{\caO}(\caC)$ be a diagram of underlying colorwise $\caK$-colocal algebras. 
Then $\hocolim_{\mathcal{I}} D$ is underlying colorwise $\caK$-colocal.
\end{lemma}
\begin{proof}
We may assume $\caO$ has an underlying cofibrant collection because $\caC$ is a simplicial symmetric monoidal model category. 
Also assume without loss of generality that $D$ takes values in cofibrant objects. 
For every $i\in \mathcal{I}$, 
let $D(i)_\bullet \to D(i)$ be the augmented standard simplicial object associated to $D(i)$. 
Note that by Proposition \ref{proposition:rcn}(i) and the explicit formula for the free functor $F$, 
$U(FU)^n D(i)$ is cofibrant for every $i\in\mathcal{I}$ and $n \ge 0$. 
For $X_n= \hocolim_I D(-)_n$ we have $X_n\simeq \bL F(\hocolim_I U(FU)^n D(-))$. 
By Lemma~\ref{lemma:underlyingcolocal}, each $U(FU)^n D(i)$ is colorwise $\caK$-colocal and thus $X_n$ is underlying colorwise $\caK$-colocal.
Lemma~\ref{lemma:hocolimmap} implies that $\hocolim_ID \simeq\hocolim_{\Delta^{\rm op}} X_\bullet$. 
Finally, 
$\hocolim_{\Delta^{\rm op}}U(X_\bullet) \simeq U(\hocolim_{\Delta^{\rm op}} X_\bullet)$ follows from Lemma~\ref{lemma:hocolimUweakequivalence}.
\end{proof}

\begin{proposition} 
\label{proposition:colocalunderlying}
With the same assumptions as in Lemma \ref{lemma:underlyingcolocal} the following holds.
\begin{itemize}
\item[{\rm (i)}] If $\cA$ is an  underlying colorwise $\caK$-colocal $\caO$-algebra, then $\cA$ is $\bL F(\caK^C)$\nobreakdash-colocal.
\item[{\rm (ii)}] If $\Alg_{\caO}(\caC)$ has a good $\bL F(\caK^C)$-colocalization, then every $\bL F(\caK^C)$-colocal object is underlying colorwise $\caK$-colocal.
\end{itemize}
\end{proposition}
\begin{proof}
Again we may assume that $\caO$ has an underlying cofibrant collection since $\caC$ is a simplicial symmetric monoidal model category.

To prove part (i) we may assume $\cA$ is cofibrant. 
Let $\cA_\bullet \to \cA$ be the associated augmented standard simplicial object. 
As in the proof of Lemma \ref{lemma:hocolimunderlying} it follows that $\cA_n$ has the correct homotopy type for every $n$, 
i.e., each $\cA_n$ is weakly equivalent to $((\bL F) U)^{n+1}(\cA)$. 
By Lemma \ref{lemma:hocolimmap}, $\hocolim \cA_\bullet \to \cA$ is a weak equivalence. 
Each $\cA_n$ is $F(\caK^C)$-colocal by Lemmas~\ref{lemma:underlyingcolocal} and~\ref{lemma:general}(ii).
Thus $\cA$ is $F(\caK^C)$-colocal.

For part (ii), note that if $X$ in $\caC^C$ is colorwise $\caK$-colocal,
then $\bL F(X)$ is underlying colorwise $\caK$-colocal by Lemma \ref{lemma:underlyingcolocal}. 
We conclude from Lemma \ref{lemma:hocolimunderlying} since by assumption the $F(\caK^C)$-colocal objects are generated under homotopy colimits by $F(\caK^C)$.
\end{proof}

\begin{theorem} 
\label{theorem:colocalcommute}
Let $\caO$ be a strongly admissible $C$-colored operad in a cofibrantly generated simplicial symmetric monoidal model category $\caC$. 
Let $\caK$ be a tensor-closed set of isomorphism classes of objects of $\Ho(\caC)$.
Suppose $\Alg_{\caO}(\caC)$ has a good $\bL F(\caK^C)$-colocalization and $\caO(c_1,\ldots, c_n;c)\otimes -$ preserves $\caK$\nobreakdash-colocal objects for all
$(c_1,\ldots,c_n,c)$, $n\ge 0$.
If $\cA'\to \cA$ is an $\bL F(\caK^C)$-colocalization of $\cA$ in $\Alg_{\caO}(\caC)$, 
then $U(\cA') \to U(\cA)$ is a $\caK^C$\nobreakdash-colocalization in $\caC^C$.
\end{theorem}
\begin{proof}
By Proposition \ref{proposition:colocalunderlying}(ii) the object $U(\cA')$ is $\caK^C$-colocal, and by Lemma \ref{lemma:general}(ii) the map $U(\cA') \to U(\cA)$ is a $\caK^C$-colocal equivalence.
\end{proof}

\begin{remark}
Theorem \ref{theorem:colocalcommute} implies Proposition \ref{proposition:colocalunderlying}(i) provided $\Alg_{\caO}(\caC)$ acquires a good $\bL F(\caK^C)$-colocalization. 
If $\caC^C$ has a good $\caK^C$-colocalization, 
the theorem states that for a cofibrant replacement $\cA'\to \cA$ in $\Alg_{\caO}(\caC)^{\bL F(\caK^C)}$ the map $U(\cA')\to U(\cA)$ is a cofibrant replacement in $(\caC^C)^{\caK^C}$.
\end{remark}

\begin{proposition}
If either of the model structures $\Alg_{\caO}(\caC)^{\bL F(\caK^C)}$ or $\Alg_{\caO}(\caC^{\caK})$ exists,
then so does the other and they coincide.
\end{proposition}
\begin{proof}
It suffices to check that the fibrations and weak equivalences coincide. 
For the fibrations, note that the model structures on the algebras are transferred from $\caC$ and $\caC^{\caK}$, for the same classes of fibrations. 
For the weak equivalences we use Lemma~\ref{lemma:general}(ii). 
\end{proof}

\begin{remark}
The model structure $\Alg_{\caO}(\caC)^{\bL F(\caK^C)}$ exists if $\Alg_{\caO}(\caC)$ is right proper, 
e.g., when $\caC$ is right proper.
We also remark that $\Alg_{\caO}(\caC^{\caK})$ exists if the colocalized model structure $\caC^{\caK}$ can be transferred. 
\end{remark}

\subsection{Localization of algebras}
In this section we show that tensor-closed $\caS$\nobreakdash-localization functors preserve algebras over cofibrant $C$-colored operads. 
More precisely, we prove that if $\caO$ is a cofibrant $C$-colored operad and $f$ an $\bL F(\caS)$-localization in $\Alg_{\caO}(\caC)$,
then $U(f)$ is an $\caS$\nobreakdash-localization in $\caC$. 

If $\caS$ is a set of homotopy classes of maps in $\caC$ and $C$ is a set of colors, we denote by $\caS^C$ the set $\prod_{c\in C}\caS$.
Note that a map in $\caC^C$ is an $\caS^C$-local equivalence if and only if it is colorwise an $\caS$-local equivalence.

\begin{lemma}
\label{lemma:FU_loc}
Let $\caO$ be a strongly admissible $C$-colored operad in a cofibrantly generated simplicial symmetric monoidal model category $\caC$.
Suppose $\caS$ is set of homotopy classes of maps such that $\caS$-equivalences are tensor-closed.
If $g$ is colorwise an $\caS$-equivalence, then $F(g)$ is underlying colorwise an $\caS$-equivalence.
\end{lemma}
\begin{proof}
Let $g\colon \cA\to \cB$ be a map in $\caC$. Then $UF(g)$ is the map 
$$
\xymatrix{
\coprod\limits_{n\ge 0}\left(\coprod\limits_{c_1,\dots, c_n\in C}\mathcal{O}(c_1,\ldots, c_n;c)
\otimes_{\Sigma_n} \cA(c_1)\otimes\cdots\otimes \cA(c_n)\right) \ar[d] \\
\coprod\limits_{n\ge 0}\left(\coprod\limits_{c_1,\dots, c_n\in C}\mathcal{O}(c_1,\ldots, c_n;c)
\otimes_{\Sigma_n} \cB(c_1)\otimes\cdots\otimes \cB(c_n)\right).
}
$$
By assumption, the map $\cA(c_1)\otimes\cdots\otimes \cA(c_n)\to \cB(c_1)\otimes\cdots\otimes \cB(c_n)$ is an $\caS$-local equivalence for every $n$-tuple $(c_1,\ldots, c_n)$,
and tensoring with $\caO(c_1,\ldots, c_n)$ preserves this property.
The result follows by using that $\caS$-local equivalences are closed under homotopy colimits and coproducts.
\end{proof}
\begin{remark}
The assumptions of the theorem are automatically satisfied if, for instance, the functor $X\otimes^{\bL}-$ preserve $\caS$-local equivalences for all $X$ in $\caC$.
\end{remark}
\begin{theorem}
\label{theorem:lkjoij}
Let $\caO$, $C$, $\caC$, and $\caS$  be as above and suppose in addition that $\Alg_{\caO}(\caC)$ has a good $\bL F(\caS^C)$-localization.
If $\cA\to \cA'$ is an $\bL F(\caS^C)$-localization in $\Alg_{\caO}(\caC)$, then $U(\cA) \to U(\cA')$ is an $\caS^C$\nobreakdash-localization in $\caC^C$.
\end{theorem}
\begin{proof}
By Lemma~\ref{lemma:general}(i) it follows that $U(\cA')$ is $\caS^C$-local. 
It remains to show the map $U(\cA)\to U(\cA')$ is an $\caS^C$-local equivalence. 
Consider the diagram
$$
\xymatrix{
FU\cA\ar[d] & FUFU\cA\ar[d]\ar[l] &\ar[l]\cdots &(FU)^n\cA\ar[d] \ar[l] &\ar[l]\cdots \\
F(\widehat{U\cA}) & FUF(\widehat{U\cA})\ar[l] &\ar[l]\cdots &(FU)^{n-1}F(\widehat{U\cA}) \ar[l] &\ar[l]\cdots, 
}
$$
where $U\cA\to \widehat{U\cA}$ is a fibrant replacement of $U\cA$ in the localized model category $(\caC^C)_{\caS^C}$. 
The leftmost vertical map is an $F(\caS^C)$-local equivalence by Lemma~\ref{lemma:general}(i). 

By~\cite[Theorem 5.7]{Gut12} the map $U\cA\to \widehat{U\cA}$ coincides with $U(\cA\to \cB)$ for some map of $\caO$-algebras $\cA\to \cB$. 
Lemma~\ref{lemma:FU_loc} shows $UFU\cA\to UFU\cB$ is an $\caS^C$-local equivalence, 
hence $FUFU\cA\to FUFU\cB$ is an $F(\caS^C)$-local equivalence. 
Iterating this argument, it follows that
$$
\cA_n=(FU)^n\cA\longrightarrow (FU)^{n-1}F(\widehat{U\cA})=(FU)^n \cB=\cB_n
$$
is an $F(\caS^C)$-local equivalence. 
Taking homotopy colimits in the previous diagrams yields the commutative square
$$
\xymatrix{
\cA\ar[d] & \hocolim_{\Delta^{op}} \cA_{\bullet} \ar[l]\ar[d] \\
\cB & \hocolim_{\Delta^{op}} \cB_{\bullet}. \ar[l]
}
$$
The right vertical map is an $F(\caS^C)$-local equivalence (a homotopy colimit of $F(\caS^C)$-local equivalences).
The horizontal maps are weak equivalences by Lemma~\ref{lemma:hocolimmap}. 
Hence $\cA\to \cB$ is an $F(\caS^C)$-local equivalence. By repeating the same construction with $\cA'$ instead of $\cA$, we obtain a commutative diagram
$$
\xymatrix{
\cA \ar[r] \ar[d] & \cA'\ar[d] \\
\cB \ar[r] & \cB', 
}
$$
where all four maps are $F(\caS^C)$-local equivalences and $\cA'$, $\cB$ and $\cB'$ are $F(\caS^C)$-local. Hence the left vertical map and the bottom horizontal map are weak equivalences. If we apply the forgetful functor we get a commutative diagram
$$
\xymatrix{
U(\cA) \ar[r] \ar[d] & U(\cA')\ar[d] \\
U(\cB)=\widehat{U(\cA)} \ar[r] & \cB'=\widehat{U(\cA')}, 
}
$$
where the left vertical map is an $\caS^C$-local equivalence and the right vertical and bottom horizontal maps are weak equivalences.
It follows that $U(\cA)\to U(\cA')$ is an $\caS^C$-local equivalence.
\end{proof}

\begin{proposition}
\label{prop:modelstr_localg}
If either of the model structures $\Alg_{\caO}(\caC)_{\bL F(\caS^C)}$ or $\Alg_{\caO}(\caC_{\caS})$ exists,
then so does the other and they coincide.
\end{proposition}
\begin{proof}
It suffices to check that the trivial fibrations and fibrant objects coincide. 
For the trivial fibrations, 
note that the model structures on the algebras are transferred from $\caC$ and $\caC_{\caS}$ for the same trivial fibrations. 
For the fibrant objects we use Lemma~\ref{lemma:general}(i). 
\end{proof}

\begin{remark}
The model structure $\Alg_{\caO}(\caC)_{\bL F(\caS^C)}$ exists if $\Alg_{\caO}(\caC)$ is left proper. 
We also remark that $\Alg_{\caO}(\caC_{\caS})$ exists if the localized model structure $\caC_{\caS}$ can be transferred. 
\end{remark}

\section{(Co)localization of modules over algebras}
\label{section:proofsofthemainresults}

In the following we shall run similar arguments for modules over a given monoid instead of algebras over a colored operad, 
culminating in analogous statements of Theorem~\ref{theorem:colocalcommute} and Theorem~\ref{theorem:lkjoij}. 
When colocalizing (resp. localizing) a module over a monoid $\cA$ with respect to a tensor-closed set of objects $\caK$ (resp. of morphisms $\caS$) for which $\cA$ is $\caK$-colocal (resp. $\caS$-local), 
one can simply apply Theorem \ref{theorem:colocalcommute} or  Theorem~\ref{theorem:lkjoij} because there exists an operad whose algebras are exactly the modules over the given monoid. 
That is, let $\caO$ be the operad with $\caO(1)=\cA$ and $\caO(i)=\emptyset$ for $i \neq 1$. 
Then the categories of $\caO$-algebras and $\cA$-modules are equivalent.
Furthermore, $\caO$ is strongly admissible if $\cA$ is. 
But in practice, 
e.g., for the motivic slice filtration, 
one wants to colocalize or localize a module with respect to a colocalization or localization functor other than the one for which the monoid is colocal or local.

\subsection{Colocalization of modules}
\label{subsection:gfdgfc}

We first address colocalization of modules over monoids, 
and second colocalization of modules over arbitrary operads.
In the latter case we employ enveloping algebras and restrict to monoids.

\begin{lemma} 
\label{lemma:gtrjg}
Let $\cA$ be a strongly admissible monoid in a symmetric monoidal model category $\caC$. 
Then the forgetful functor $U\colon \Mod(\cA) \to \caC$ preserves homotopy colimits, where $\Mod(\cA)$ denotes the category of $\cA$-modules.
\end{lemma}
\begin{proof}
Since $\cA$ is strongly admissible, 
we may assume its underlying object is cofibrant (case (i) in Definition~\ref{def:strongly_admissible}), or its unit map is a cofibration in $\caC$ (case (ii)). 
In the first case, 
$U$ is a left Quillen functor since its right adjoint given by the internal hom $\Hom(\cA,-)$ preserves fibrations and trivial fibrations, 
so the result follows. 
In the second case, 
the same argument shows that $U$ is a left Quillen functor for the model structure on $\caC$ furnished by the strong admissibility of $\cA$.
\end{proof}

\begin{proposition} 
\label{proposition:colocalunderlying2}
With $\cA$ and $\caC$ as in Lemma \ref{lemma:gtrjg}, 
let $\caK$ be a set of isomorphism classes of objects for $\Ho(\caC)$. 
Suppose $\cA \otimes^\bL \caK$ is underlying $\caK$-colocal.
If $\cM \in \Mod(\cA)$ is underlying $\caK$-colocal, then it is also $\cA \otimes^\bL \caK$-colocal.
If, in addition, $\Mod(\cA)$ has a good $\cA \otimes^\bL \caK$-colocalization, then every $\cA \otimes^\bL \caK$-colocal $\cA$-module is underlying $\caK$-colocal.
\end{proposition}
\begin{proof}
We may assume $\cM$ is cofibrant.
It follows, 
using the left Quillen functor $U$ in the proof of Lemma \ref{lemma:gtrjg},
that $\cM$ is underlying cofibrant (in case I), or cofibrant in $\caC$ for the second model structure (in case II). 
Letting $\cM_n=\cA^{\otimes (n+1)} \otimes \cM$, the augmented simplicial $\cA$-module $\cM_\bullet \to \cM$ splits after forgetting the $\cA$-module structure. 
By Lemma \ref{lemma:gtrjg} the natural $\cA$-module map $\hocolim \cM_\bullet \to \cM$ is a weak equivalence. 
Each $\cM_n$ is $\cA \otimes^\bL \caK$-colocal by Lemma \ref{lemma:general}(ii) and the assumption that $\cA \otimes^\bL \caK$ is underlying $\caK$-colocal, since $\caK$-colocal object are generated by taking the closure of $\caK$ under weak equivalences and homotopy colimits. It follows that $\cM$ is $\cA \otimes^\bL \caK$-colocal.

For the second assertion, we use Lemma \ref{lemma:gtrjg}, 
the fact that $\cA\otimes^\bL \caK$-colocal $\cA$-modules are generated under homotopy colimits by $\cA\otimes^\bL \caK$, 
and the assumption that $\cA \otimes^\bL \caK$ is underlying $\caK$-colocal.
\end{proof}

\begin{remark}
Note that since we are dealing with monoids instead of arbitrary operads, we do not assume in Proposition~\ref{proposition:colocalunderlying2} that the set $\caK$ is tensor-closed (cf. Proposition~\ref{proposition:colocalunderlying}).
\end{remark}
\begin{theorem} 
\label{theorem:bhtgrfg}
With $\cA$, $\caC$, and $\caK$ as in Proposition \ref{proposition:colocalunderlying2}, 
suppose that $\cA \otimes^\bL \caK$ is underlying $\caK$-colocal and $\Mod(\cA)$ has a good $\cA \otimes^\bL \caK$-colocalization. 
If $\cM' \to \cM$ is a $\cA \otimes^\bL \caK$-colocalization of $\cM \in \Mod(\cA)$, then $U\cM'\to U\cM$ is a $\caK$-colocalization in $\caC$.
\end{theorem}
\begin{proof}
Proposition \ref{proposition:colocalunderlying2} implies that $\cM'$ is underlying $\caK$-colocal. 
Using Lemma \ref{lemma:general}(ii) we conclude that $\cM' \to \cM$ is an underlying $\caK$-colocal equivalence.
\end{proof}

Next we discuss $E_\infty$ operads,
i.e., parameter spaces for multiplication maps that are associative and commutative up to all higher homotopies, 
and their algebras.
For an operad $\caO$ in $\caC$ and an $\caO$-algebra $\cA$ we denote by $\Env_\caO(\cA)$ the enveloping algebra of $\cA$. 
This is a monoid with the property that $\Mod(\cA)\simeq \Mod(\Env_\caO(\cA))$.
For an operad $\caO$ with underlying cofibrant collection, 
$\Alg_{\caO}(\caC)$ has a left semi model structure \cite[Theorem 4.7]{Spi} provided the domains of the generating cofibrations of $\caC$ are small relative to the whole category.

\begin{theorem} 
\label{theorem:bhgzhj}
Let $\caC$ be a cofibrantly generated symmetric monoidal model category with a set $\caK$ of isomorphism classes of objects for $\Ho(\caC)$.
Suppose $\caC$ is left proper, 
its generating cofibrations can be chosen in such a way that their domains are cofibrant and small relative to the whole category, 
and its unit is cofibrant. 
Let $\caO$ be a pointed $E_\infty$ operad in $\caC$. 
Suppose $\cA \in \Alg_{\caO}(\caC)$ is cofibrant, the objects of $\cA \otimes^\bL \caK$ are underlying $\caK$-colocal, and $\Mod(\cA)$ has a good $\Env_\caO(\cA)\otimes^\bL \caK$-colocalization. If $\cM' \to \cM$ is an $\Env_\caO(\cA) \otimes^\bL \caK$-colocalization of $\cM \in \Mod(\cA)$, then $U\cM'\to U\cM$ is a $\caK$-colocalization in $\caC$.
\end{theorem}
\begin{proof}
The enveloping algebra $\Env_\caO(\cA)$ is underlying cofibrant in $\caC$ \cite[Corollary 6.6]{Spi} (the cofibrancy assumption on the unit is missing in loc.~cit.). 
By \cite[Lemma 8.6]{Spi},
the adjoint $\Env_\caO(\cA) \to \cA$ in $\Mod(\cA)$ of the unit map for $\cA$ is a weak equivalence (here we use that $\caO$ is an $E_\infty$ operad). 
Hence the $\cA$-module $\Env_\caO(\cA) \otimes^\bL \caK$ is underlying $\caK$-colocal. 
Thus $\Env_\caO(\cA)$ satisfies the assumptions of Theorem \ref{theorem:bhtgrfg}, and the result follows.
\end{proof}

\begin{remark}
In the above theorem we could also assume that $\caO$ is cofibrant as an operad (the operads in $\caC$ form a left semi model category over $\caC^{\Sigma,\bullet}$ 
--- for notation, see \cite[\S 3]{Spi} --- by \cite[Theorem 3.2]{Spi}), and $\cA$ is underlying cofibrant \cite[Corollaries 6.3, 8.7]{Spi}.
\end{remark}

It is desirable to have a parallel theory for modules over operad algebras (in the one-colored case). Since we have the equivalence $\Mod(\cA)\simeq \Mod(\Env_\caO(\cA))$ and the enveloping algebra is always a monoid, we can restrict to the latter case. 
A key point is to show that $\Env_\caO(\cA)$ is underlying $\caK$-colocal under suitable assumptions, 
making our proof of Theorem \ref{theorem:bhgzhj} for $E_\infty$ operads go through.
For this we employ the simplicial resolution $\cA_\bullet \to \cA$. 
It is easily seen that $\Env_\caO(\cA_n)$ is underlying $\caK$-colocal for each $n \ge 0$, 
so the result follows provided $\Env_\caO(\cA)$ is weakly equivalent to the homotopy colimit over $\Delta^\mathrm{op}$ of the diagram $\Env_\caO(\cA_\bullet)$.

For a symmetric monoidal category $\caC$, 
we denote by $\Pairs(\caC)$ the category of pairs $(\caO,\cA)$,
where $\caO\in\Oper(\caC)$ and $\cA\in\Alg_{\caO}(\caC)$.
Next we review some facts about the colored operads $\cOper$ and $\cPairs$ whose algebras are $\Oper(\caC)$ and $\Pairs(\caC)$, respectively. 
The set of colors for $\cOper$ is $\N$, while for $\cPairs$ it is $\N \cup \{a\}$.
The operad $\cOper$ is a special case of a colored operad defined in \cite[\S3]{GV12} whose algebras are itself colored operads for a fixed set of colors $C$. 
We take $C$ to be a one point set and let $\cOper=S_\caC^C$ in the notation of  \cite{GV12}. 
The colored operad $\cOper$ is the image in $\caC$ of an $\N$-colored operad in sets denoted $S^C$, 
which we now describe (an explicit description of $\cOper$ can be found in~\cite[1.5.6]{BM07}).

Let $S^C(n_1,\ldots,n_k;n)$ denote the set of isomorphism classes of certain trees.  
We consider planar connected directed trees such that each vertex has exactly one outgoing edge. 
There are two different types of edges, 
namely inner edges with vertices at both ends, and external edges with a vertex only at one end or no vertices at all. 
It follows that there is exactly one external edge leaving a vertex, the so-called root. 
There are $n$ external edges which are input edges to vertices, called leaves. 
These are numbered by $\{1,\ldots,n\}$. There are $k$ vertices numbered by $\{1,\ldots,k\}$. 
The planarity of the tree means that the input edges of each vertex $v$ are numbered by $\{1,\ldots,\mathrm{in}(v)\}$, 
where if $v$ is numbered by $i$, then $\mathrm{in}(v)=n_i$.
As described in \cite[1.5.6]{BM07} or \cite[\S3.2]{GV12} there is an $\N$-colored operad structure on $S^C$. 
We set $\cOper^s=S^C$. 
Then $\cOper$ is the image of $\cOper^s$ under the tensor functor sending the one point set to the unit, 
and $\Alg(\cOper)\simeq\Oper(\caC)$ \cite[Proposition 3.5, \S3.3]{GV12}.

Let $c_1,\ldots,c_k \in \N \cup \{a\}$, $n \in \N$. 
If each $c_i$ is in $\N$, we set $\cPairs(c_1,\ldots,c_k;n)=\cOper(c_1,\ldots,c_k;n)$, and otherwise we set $\cPairs(c_1,\ldots,c_k;n)=\emptyset$. 
If the output $c=a$, then $\cPairs(c_1,\ldots, c_k; c)=\cOper(c'_1,\ldots, c'_k; 0)$, where $c'_i=c_i$ if $c_i\in \N$ and $c'_i=0$ if $c_i=a$.

\begin{proposition} 
\label{proposition:nhgdes}
There is an $\N \cup \{a\}$-colored operad structure on $\cPairs$.
Moreover, 
there is a natural equivalence $\Alg(\cPairs) \simeq \Pairs(\caC)$.
\end{proposition}
\begin{proof} The composition product and the unit maps of $\cPairs$ are defined using the composition product and unit maps of $\cOper$. 
If $\cA$ is a $\cPairs$-algebra, then the structure maps
$$
\cPairs(c_1,\ldots, c_k;c)\otimes\cA(c_1)\otimes\cdots\otimes  \cA(c_k)\longrightarrow \cA(c)
$$
when $c_1,\ldots, c_k, c\in \N$ give the sequence $\caO=\{\cA(n)\}_{n\ge 0}$ the structure of an operad, since $\cPairs(c_1,\ldots, c_k;c)=\cOper(c_1,\ldots, c_k; c)$. The $\caO$-algebra structure on $\cA(a)$ is defined by the structure maps
$$
\cPairs(n, a,\stackrel{(k)}{\ldots}, a; a)\otimes \cA(k)\otimes \cA(a)\otimes\stackrel{(n)}{\cdots}\otimes \cA(a)\longrightarrow \cA(a),
$$
since $\cPairs(n, a,\stackrel{(k)}{\ldots}, a; a)=\cOper(n, 0,\stackrel{(n)}{\ldots}, 0;0)$.
\end{proof}

Note that, as for $\cOper$, $\cPairs$ is the image in $\caC$ of a colored operad, say $\cPairs^s$, in sets.

\begin{lemma} 
\label{lemma:nhgfsx}
If the unit in $\caC$ is cofibrant, then the underlying collections of $\cOper$ and $\cPairs$ are cofibrant. 
More precisely, 
let $c_1,\ldots,c_k$ and $c$ be sequences of colors for $\cOper$ and $\cPairs$, respectively.
Then the stabilizer groups of these sequences 
--- which are subgroups of $\Sigma_k$ ---
act freely on $\cOper^s(c_1,\ldots,c_k;c)$ and $\cPairs^s(c_1,\ldots,c_k;c)$, respectively.
\end{lemma}
\begin{proof}
This uses the explicit description of these colored operads:
two isomorphic planar trees of the type we consider are already uniquely isomorphic, 
the additional numbering of the vertices --- and leaves for the case of $\cPairs^s$ --- force the actions to be free.
\end{proof}

\begin{proposition} 
\label{proposition:bfgdjf}
Let $\caC$ be a cofibrantly generated simplicial symmetric monoidal model category such that all of its objects are small relative to the whole category.
Suppose $\cPairs$ is strongly admissible (e.g., the unit in $\caC$ is cofibrant and $\Pairs(\caC)$ has a transferred model structure by Lemma \ref{lemma:nhgfsx}). 
For a simplicial object $\cA_\bullet$ in $\Pairs(\caC)$ there is a canonical weak equivalence
$$
\hocolim_{\Delta^\mathrm{op}} U(\cA_\bullet) \longrightarrow U(\hocolim_{\Delta^\mathrm{op}} \cA_\bullet).
$$
Here, 
$U$ denotes the forgetful functor $\Pairs(\caC)\to \caC^{\N \cup \{a\}}$.
\end{proposition}
\begin{proof}
This follows directly from Lemma \ref{lemma:hocolimUweakequivalence} and Proposition \ref{proposition:nhgdes}.
\end{proof}

There is an embedding $\phi \colon\Oper(\caC) \to \Pairs(\caC)$ given by $\caO \mapsto (\caO,\caO(0))$. 
It is shown in \cite[Proposition 1.6]{Berger-Moerdijk.Der} that $\phi$ has a left adjoint $(\caO,\cA) \mapsto \caO_\cA$. 
The operad $\caO_\cA$ has the property that the category of $\caO$-algebras under $\cA$ is equivalent to $\caO_\cA$-algebras, 
and the canonical $\caO$-algebra map $\cA\to \caO_\cA(0)$ is an isomorphism \cite[Lemma 1.7]{Berger-Moerdijk.Der}.
Moreover, there is a canonical isomorphism of monoids $\Env_\caO(\cA) \cong\caO_\cA(1)$; see \cite[Theorem 1.10]{Berger-Moerdijk.Der}.

\begin{lemma} 
\label{lemma:bgfeirg}
Suppose $\caC$ is a symmetric monoidal model category, and $\Oper(\caC)$ and $\Pairs(\caC)$ have transferred model structures. 
Then the embedding $\phi \colon \Oper(\caC) \to \Pairs(\caC)$ is a right Quillen functor.
\end{lemma}
\begin{proof}
With these assumptions the functor $\phi$ has a left adjoint and preserves fibrations and weak equivalences.
\end{proof}

\begin{lemma} 
\label{lemma:gdjfdf}
Let $\caC$ be a cofibrantly generated symmetric monoidal model category such that the domains of the generating cofibrations are small relative to the whole category. 
Suppose $\Pairs(\caC)$ has a transferred model structure.
Then $(\caO,\cA)\in \Pairs$ is cofibrant if and only if $\caO$ is cofibrant in $\Oper(\caC)$ and $\cA$ is cofibrant in $\Alg_{\caO}(\caC)$.
\end{lemma}
\begin{proof}
Recall $\Oper(\caC)$ has a left semi model structure over $\caC^{\Sigma,\bullet}$, 
and if $\caO \in \Oper(\caC)$ is cofibrant, then the same holds for $\Alg_{\caO}(\caC)$ over $\caC$ \cite[Theorems 3.2, 4.3]{Spi}.

Suppose $(\caO,\cA)$ is cofibrant. 
The lifting property with respect to trivial fibrations $(\caO_1,\mathrm{pt}) \to (\caO_2,\mathrm{pt})$ shows that $\caO$ is cofibrant. 
And the lifting property with respect to trivial fibrations $(\caO,\cA_1) \to (\caO,\cA_2)$ shows that $\cA$ is cofibrant in $\Alg_{\caO}(\caC)$.

Conversely, assume $\caO$ and $\cA$ are cofibrant. 
Let $(\caO_1,\cA_1) \to (\caO_2,\cA_2)$ be a trivial fibration in $\Pairs(\caC)$, and $(\caO,\cA) \to (\caO_2,\cA_2)$ a map. 
First, we can lift $\caO \to \caO_2$ to a map $\caO \to \caO_1$. 
Pulling the algebras $\cA_1$ and $\cA_2$ back to $\caO$ gives us a lifting problem in $\Alg_{\caO}(\caC)$, which can be solved.
\end{proof}

\begin{theorem}
\label{theorem:bgrjf}
Let $\caC$ be a cofibrantly generated simplicial symmetric monoidal model category such that all of its objects are small relative to the whole category.
Let $\caK$ be a tensor-closed set of isomorphism classes of objects for $\Ho(\caC)$. 
Suppose $\cOper$ and $\cPairs$ are strongly admissible (e.g., if the unit in $\caC$ is cofibrant and $\Oper(\caC)$ and $\Pairs(\caC)$ have transferred model structures). 
If $(\caO,\cA) \in \Pairs(\caC)$ is cofibrant and each $\caO(n)$ is $\caK$-colocal and $\cA$ is underlying $\caK$-colocal, then the enveloping algebra $\Env_\caO(\cA)$ is underlying $\caK$-colocal.
\end{theorem}
\begin{proof}
Let $F \colon \caC \rightleftarrows \Alg(\caO) \; \colon \! U$ be the free-forgetful adjunction. 
Let $\cA_\bullet \to \cA$ be the standard augmented simplicial object with $\cA_n=(FU)^{n+1} \cA$. 
Since $\Alg_{\caO}(\caC)$ is a left semi model category over $\caC$ it follows that $U(FU)^n \cA$, $n \ge 0$, is cofibrant (for $n>0$ one can also use the explicit formula for $F$). 
By Lemma \ref{lemma:gdjfdf} it follows that $(\caO,\cA_n) \in \Pairs(\caC)$ is cofibrant.

For $X \in \caC$ the enveloping algebra $\Env_\caO(FX) \cong\caO_{FX}(1)$ is given by the formula
$$
\Env_\caO(FX) \cong \bigoplus_{n \ge 0} \caO(n+1) \otimes_{\Sigma_n} X^{\otimes n}.
$$
It follows that $\Env_\caO(\cA_n)$ is underlying $\caK$-colocal for each $n \ge 0$.

Since the augmented simplicial object $U\cA_\bullet \to U\cA$ splits, Proposition \ref{proposition:bfgdjf} for $\Pairs(\caC)$ implies there is a canonical weak equivalence
$$
\hocolim_{\Delta^\mathrm{op}} \cA_\bullet \longrightarrow \cA.
$$ 
Next we apply the derived functor of the left Quillen functor $(\caO',\cA') \mapsto \caO'_{\cA'}$ --- see Lemma \ref{lemma:bgfeirg} --- to $(\caO,\cA_\bullet) \to (\caO,\cA)$, 
giving the augmented simplicial object $\caO_{\cA_\bullet} \to \caO_\cA$ in $\Oper(\caC)$. 
Since derived left Quillen functors commute with homotopy colimits, 
there is a weak equivalence
$$
\hocolim_{\Delta^\mathrm{op}} \caO_{\cA_\bullet} \longrightarrow \caO_\cA.
$$ 
Proposition \ref{proposition:bfgdjf} for $\Oper(\caC)$ implies the weak equivalence 
$$
\hocolim_{\Delta^\mathrm{op}} \caO_{\cA_\bullet}(1) \longrightarrow \caO_\cA(1).
$$
Here, 
the homotopy colimit is computed in $\caC$.
It follows that $\caO_\cA(1) \cong \Env_\caO(\cA)$ is underlying $\caK$-colocal, as claimed.
\end{proof}

\begin{corollary}
Let $\caC$, $\caO$, $\cA$, and $\caK$ be as in Theorem~\ref{theorem:bgrjf} and suppose that $\Mod(\cA)$ has a good $\Env_{\caO}(\cA)\otimes^{\bL}\caK$-colocalization. 
If $\caM'\to \caM$ is an $\Env_{\caO}(\cA)\otimes^{\bL}\caK$-colocalization of $\caM\in\Mod(\cA)$, 
then $U(\caM')\to U(\caM)$ is a $\caK$-colocalization in $\caC$.
\end{corollary}
\begin{proof}
Since $\caK$ is tensor closed and $\Env_\caO(\cA)$ is underlying $\caK$-colocal by Theorem~\ref{theorem:bgrjf}, 
it follows that the $\cA$-module $\Env_\caO(\cA)\otimes^{\bL}\caK$ is underlying $\caK$-colocal. 
To conclude we proceed exactly as in the proof of Theorem~\ref{theorem:bhtgrfg},
now with the monoid $\Env_{\caO}(\cA)$.
\end{proof}
\begin{remark}
One may ask for other hypothesis such that Theorem \ref{theorem:bgrjf} still holds. 
With $\caC$ and $\caK$ as above, suppose $\Oper(\caC)$ and $\Pairs(\caC)$ have transferred model structures. 
Suppose $\caC$ has a second simplicial model structure with the same weak equivalences and cofibrant unit.
We wish to conclude that a cofibrant underlying $\caK$-colocal $(\caO,\cA)$ yields an underlying $\caK$-colocal enveloping algebra $\Env_\caO(\cA)$. 

As a replacement for Proposition \ref{proposition:bfgdjf} we sketch an alternate argument: 
Suppose every Reedy cofibrant object $X_\bullet\in s\Pairs(\caC)$ is cofibrant in $s\caC^{\N \cup \{a\}}$ for the Reedy model structure.
Now $\cPairs^s$ --- viewed as a colored operad in $\sSets$ --- has an underlying cofibrant collection. 
Let us assume the objectwise tensor functor $\sSets \times s\caC \to s\caC$ is a Quillen bifunctor. 
Then $c \mathsf{i}\to X_\bullet$ is a cofibration in $s\Pairs(\caC)$, 
where $c \mathsf{i}$ is the constant simplicial object on the initial object $\mathsf{i}$ of $\Pairs(\caC)$, 
see \cite[Proposition 4.8]{Spi}. 
Since $\Delta^\mathrm{op}$ has cofibrant constants, it follows that $X_\bullet$ is underlying Reedy cofibrant.

The same argument works for $\Oper(\caC)$. 
Alternatively, one can use that $\Oper(\caC)$ is a left semi model category over $\caC^{\Sigma,\bullet}$. 
\end{remark}

\subsection{Localization of modules}
\label{subsection:gresgxdh}
As in the previous section, 
we first discuss localization of modules over monoids and then localization of modules over arbitrary operads.

Given a monoid $\cA$, we say that the functor $\cA\otimes^{\bL}-$ preserves $\caS$-equivalences if the tensor product of $\cA$ with any $\caS$-equivalence is an underlying $\caS$-equivalence.
\begin{theorem}
\label{thm:fun_loc_mon}
Let $\cA$ be a strongly admissible monoid in a symmetric monoidal model category $\caC$. 
Let $\caS$ be a set of homotopy classes of maps such that $\cA\otimes^{\bL}-$ preserves $\caS$-local equivalences and $\Mod(\cA)$ has a good $\cA\otimes^{\bL}\caS$-localization. 
If $\caM\to \caM'$ is an $\cA\otimes^{\bL}\caS$-localization of $\caM\in \Mod(\cA)$, 
then $U(\caM)\to U(\caM')$ is an $\caS$-localization in $\caC$.
\end{theorem}
\begin{proof}
The proof is basically the same as for Theorem~\ref{theorem:lkjoij}. 
We note the assumption of tensor-closedness on the $\caS$-local equivalences is not needed since the free $\cA$-module functor is defined by $F(X)=\cA\otimes X$ for every $X$ in $\caC$, 
and therefore $\cA_n=(FU)^{n+1}\cA\to (FU)^{n+1}\cB=\cB_n$ is an $F(\caS)$-equivalence for every map of monoids $\cA\to \cB$.  
\end{proof}

\begin{theorem}
\label{thm:loc_modules}
Let $\caC$ be a cofibrantly generated simplicial symmetric monoidal model category such that all of its objects are small relative to the whole category. 
Let $\caS$ be a set of homotopy classes of maps such that $\caS$-equivalences are tensor-closed. 
Suppose $\cOper$ and $\cPairs$ are strongly admissible (e.g., if the unit in $\caC$ is cofibrant and $\Oper(\caC)$ and $\Pairs(\caC)$ have transferred model structures). 
Let $(\caO,\cA)\in\Pairs(\caC)$ be cofibrant. 
If $\cA\otimes^{\bL}-$ preserves $\caS$-equivalences, then so does $\Env_{\caO}(\cA)\otimes^{\bL}-$.
\end{theorem}
\begin{proof}
Let $F \colon \caC \rightleftarrows \Alg(\caO) \; \colon \! U$ be the free-forgetful adjunction. 
Let $\cA_\bullet \to \cA$ be the standard augmented simplicial object with $\cA_n=(FU)^{n+1} \cA$. 
Suppose that for every $\caS$-local equivalence $g$ the map $\cA\otimes^{\bL}g$ is an $\caS$-local equivalence. 
Then, $\Env_{\caO}(\cA_n)\otimes^{\bL}g$ is also an $\caS$-local equivalence for every $n\ge 0$. 
Now, 
using the same argument as in the proof of Theorem~\ref{theorem:bgrjf} with the operad $\caO_{\cA}$, 
it follows that $\Env_{\caO}\otimes g$ is an $\caS$-equivalence.
\end{proof}

\begin{corollary}
Let $\caC$, $\caO$, $\cA$ and $\caS$ be as in Theorem~\ref{thm:loc_modules} and suppose that $\Mod(\cA)$ has a good $\Env_{\caO}(\cA)\otimes^{\bL}\caS$-localization. 
If $\caM\to \caM'$ is an $\Env_{\caO}(\cA)\otimes^{\bL}\caS$-localization for $\caM\in\Mod(\cA)$, 
then $U(\caM)\to U(\caM')$ is an $\caS$-localization in $\caC$.
\end{corollary}
\begin{proof}
Theorem~\ref{thm:loc_modules} shows $\Env_{\caO}(\cA)\otimes^{\bL}-$ preserves $\caS$-local equivalences. 
The result follows by applying Theorem~\ref{thm:fun_loc_mon} to the monoid $\Env_{\caO}(\cA)$.
\end{proof}

\appendix{
\section{Preliminaries on model categories}

If $\caC$ is a \emph{cofibrantly generated model category} with set of generating  cofibrations $I$ and set of generating trivial cofibrations $J$, 
we implicitly assume the (co)domains of the elements of $I$ are small relative to the $I$-cellular maps
and that the (co)domains of the elements of $J$ are small relative to the $J$-cellular maps. 
This condition is satisfied if $\caC$ is a \emph{combinatorial model category}; 
that is, 
$\caC$ is cofibrantly generated and locally presentable, since in this case every object is $\lambda$-small for some cardinal $\lambda$.
Let $\sSets$ denote the category of simplicial sets.
Recall that for a simplicial symmetric monoidal model category $\caC$ there exists a monoidal Quillen adjunction $i\colon\sSets\rightleftarrows \caC\colon r$. 
Any such  $\caC$ is canonically enriched and (co)tensored over $\sSets$.
The tensor, enrichment, and cotensor are defined by the functors $i(-)\otimes -$, $r(\Hom(-,-))$, and $\Hom(i(-), -)$, respectively, 
where $\Hom(-,-)$ denotes the internal hom in $\caC$.
These three functors form a Quillen adjunction of two variables.

\subsection{The Reedy model structure on simplicial objects}
\label{hegvfkehvck}
Let $\caC$ be a model category. 
The \emph{simplicial objects in $\caC$} is the category $s\caC$ of $\Delta^{\mathrm{op}}$-diagrams in $\caC$, 
where $\Delta$ denotes the simplicial category. 
In its Reedy model structure \cite[15.3]{Hirschhorn} the weak equivalences are the levelwise weak equivalences, 
while the cofibrations and fibrations are defined by means of latching and matching objects, respectively.

Let $\caC$ be a simplicial model category.
The \emph{realization} $|X_{\bullet}|_\caC$ of a simplicial object $X_{\bullet}\colon\Delta^{\mathrm{op}}\to \caC$ is the coequalizer of the diagram
$$
\xymatrix{
\coprod_{[m]\to [n]} \Delta[m]\otimes X_n\ar@<-0.2ex>[r]  \ar@<1ex>[r] & \coprod_{[n]} \Delta[n]\otimes X_n
}
$$
induced by $X_n\to X_m$ and $\Delta[m]\to \Delta[n]$, respectively, for each map $[m]\to [n]$ in~$\Delta$. 
Using coend notation, as in \cite[18.3.2]{Hirschhorn} and \cite[IX.6]{MacLane}, this can be recast as
$$
|X_{\bullet}|_{\caC}= \int^{[n]\in\Delta} \Delta[n]\otimes X_n=\Delta \otimes_{{\Delta}^{\mathrm{op}}}X_{\bullet}.
$$
If the category is clear from the context we write $|X_\bullet|$ instead of $|X_\bullet|_\caC$.

\begin{lemma} 
\label{lemma:ngfdv}
Let $\caC$ be a simplicial model category and $X_\bullet$ a Reedy cofibrant simplicial object in $\caC$. 
Then the Bousfield--Kan map
$$
\hocolim_{\Delta^{\mathrm{op}}}X_{\bullet}
=
N(-\downarrow \Delta^{\mathrm{op}})^{\mathrm{op}}\otimes_{\Delta^{\mathrm{op}}}X_{\bullet}
\longrightarrow 
\Delta\otimes_{\Delta^{\mathrm{op}}}X_{\bullet}=|X_\bullet|
$$
is a weak equivalence.
\end{lemma}
\begin{proof}
See \cite[Theorem 18.7.4]{Hirschhorn}.
\end{proof}

The category $s^2\caC$ of \emph{bisimplicial objects in $\caC$} is the category of simplicial objects in $s\caC$. 
There is an obvious diagonal functor ${\rm diag}\colon s^2\caC\to s\caC$ defined by ${\rm diag}(X_{\bullet,\bullet})_n= X_{n, n}$.
\begin{lemma}
\label{lemma:dfhfghn} 
Let $X_{\bullet, \bullet}$ be a bisimplicial object in a simplicial model category $\caC$. 
Then there is a natural isomorphism 
$$
\int^{[n],[m] \in \Delta\times\Delta} (\Delta[n] \times \Delta[m]) \otimes X_{n,m} 
\cong
\int^{[n] \in \Delta} \Delta[n] \otimes X_{n,n}.
$$
\end{lemma}
\begin{proof}
The left Kan extension of the Yoneda functor $\Delta\to \sSets$ along the diagonal $\Delta\to\Delta\times\Delta$ is the functor $\Delta\times\Delta\to \sSets$ that sends $([n], [m])$ to 
$\Delta[n]\times\Delta[m]$. 
Hence the coends $\Delta\otimes_{\Delta^{\mathrm{op}}} {\rm diag}(X_{\bullet,\bullet})$ and $(\Delta\times\Delta)_{\Delta^{\mathrm{op}}\times\Delta^{\mathrm{op}}}X_{\bullet,\bullet}$ are 
isomorphic.
\end{proof}

If $\caC$ has a symmetric monoidal structure, there is a symmetric monoidal tensor product in $s\caC$ defined by the objectwise tensor product,
i.e., 
$(X_{\bullet}\otimes Y_{\bullet})_n=X_n\otimes Y_n$.

\begin{lemma}
\label{lemma:fgngd} 
\label{lemma:dfgfg} 
Let $\caC$ be a symmetric monoidal model category. 
\begin{itemize}
\item[{\rm (i)}] Then $s\caC$ is a symmetric monoidal model category for the Reedy model structure. 
\item[{\rm (ii)}] If $\caC$ is simplicial the realization functor is symmetric monoidal.
\end{itemize}
\end{lemma}
\begin{proof}
The first part is an application of \cite[Theorem 3.51 and Example 3.52]{Barwick}. For the second part, observe that
$$
|X_\bullet| \otimes |Y_\bullet| \cong \left(\int^{[n] \in\Delta} \Delta[n] \otimes X_n \right)
\otimes \left(\int^{[m] \in\Delta} \Delta[m] \otimes Y_m \right)
$$
$$
\cong\int^{([n],[m]) \in \Delta \times \Delta} (\Delta[n] \times \Delta[m]) \otimes X_n \otimes Y_m
\cong \int^{[n] \in \Delta} \Delta[n] \otimes X_n \otimes Y_n,
$$
where the last isomorphism follows by applying Lemma~\ref{lemma:dfhfghn} to the bisimplicial object $(X\otimes Y)_{n,m}=X_n\otimes Y_m$.
\end{proof}

\begin{lemma} 
\label{lemma:nbgfgfb}
Let $\caC$ be a cofibrantly generated model category. 
Then the Reedy model structure on $s\caC$ is cofibrantly generated.
\end{lemma}
\begin{proof}
Here we make use of smallness of the (co)domains of the sets of generating (trivial) cofibrations, 
see \cite[Theorem 15.6.27]{Hirschhorn}. 
\end{proof}

\subsection{Bousfield (co)localizations}
Let $\caC$ be a simplicial model category, 
$\mathcal{S}$ a set of homotopy classes of maps in $\caC$, 
and $\mathcal{K}$ a set of isomorphism classes of objects of $\Ho(\caC)$. 
The homotopy type of the derived simplicial mapping space  $\map(X,Y)$ can be computed using $\Map(QX, RY)$,  
where $\Map(-,-)$ is the simplicial enrichment.
Here, 
$Q$ and $R$ denote cofibrant and fibrant replacement functors in $\caC$,
respectively.

An object $Z$ in $\Ho(\caC)$ is \emph{$\caS$-local} if for every representative $f\colon A\to B$ of an element of $\caS$, 
the induced map
$$
f^*\colon \map(B, Z)\longrightarrow\map(A, Z)
$$
is an isomorphism in $\Ho(\sSets)$. 
An object $Z$ in $\caC$ is $\caS$-local if its image in $\Ho(\caC)$ is so. 
The class of $\caS$-local objects is closed under homotopy limits.
A map $g\colon X\to Y$ in $\Ho(\caC)$ is an \emph{$\caS$-local equivalence} or simply an \emph{$\caS$-equivalence} if for every $\caS$-local $Z$, the induced map
$$
g^*\colon \map(Y, Z)\longrightarrow \map(X, Z)
$$
is an isomorphism in $\Ho(\sSets)$. 
A map $X\to Y$ in $\caC$ is an $\caS$-local equivalence if its image in $\Ho(\caC)$ is so. 

A map $f\colon X\to Y$ in $\Ho(\caC)$ is a \emph{$\caK$-colocal equivalence} if for any representative $K$ of an element of $\mathcal{K}$, the induced map
$$
f_*\colon \map(K, X)\longrightarrow \map(K, Y)
$$
is an isomorphism in $\Ho(\sSets)$. 
Likewise, 
a map in $\caC$ is a $\caK$-colocal equivalence if its image in $\Ho(\caC)$ is so. 
An object $W$ in $\Ho(\caC)$ is called \emph{$\caK$-colocal} if for every $\caK$\nobreakdash-colocal equivalence $g\colon X\to Y$, 
there is an induced isomorphism 
$$
g_*\colon \map(W, X)\longrightarrow\map(W, Y)
$$
in $\Ho(\sSets)$. 
An object $W$ in $\caC$ is $\caK$-colocal if its image in $\Ho(\caC)$ is so. 
The class of $\caK$-colocal objects is closed under homotopy colimits.

If $X$ is an object of $\caC$, an \emph{$\caS$-localization} is an $\caS$-local equivalence $X\to X'$ for $X'$ $\caS$-local. 
Dually, a \emph{$\caK$-colocalization} is a $\caK$-colocal equivalence $X'\to X$ for $X'$ $\caK$-colocal.

A simplicial symmetric monoidal model category $\caC$ is \emph{tensor-closed} if the class of $\caS$-local equivalences is closed under the derived tensor product. 
Likewise, $\caK$ is \emph{tensor\nobreakdash-closed} if the class of $\caK$-colocal objects is closed under the derived tensor product.

\begin{definition}
Let $\caS$ be a set of maps and $\caK$ be a set of objects in a simplicial model category $\caC$.
\begin{itemize}
\item[{\rm (i)}] $\caC$ has a \emph{good $\caS$-localization} if the left Bousfield localization with respect to $\caS$ exists; that is, if the classes of cofibrations 
in $\caC$ and $\caS$-local equivalences define a model structure on $\caC$. 
This is the \emph{$\caS$-local model structure} denoted by $\caC_{\caS}$.
\item[{\rm (ii)}] $\caC$ has a \emph{good $\caK$-colocalization} if the right Bousfield localization with respect to $\caK$ exists; that is, if the classes of fibrations 
in $\caC$ and $\caK$-colocal equivalences define a model structure on $\caC$, and the $\caK$-colocal objects are generated under homotopy colimits by the objects of $\caK$. 
This is the \emph{$\caK$-colocal model structure} denoted by $\caC^{\caK}$.
\end{itemize}
\end{definition}

The $\caS$-local fibrations are the maps in $\caC$ with the right lifting property with respect of all maps of $\caC$ that are cofibrations and $\caS$-local equivalences, 
Similarly, the $\caK$-colocal cofibrations are the maps in $\caC$ with the left lifting property with respect to all maps of $\caC$ that are fibrations and $\caK$-colocal equivalences.

If $\caC$ has a good $\caC_{\caS}$-localization, then an $\caS$-localization of $X$ is just a fibrant replacement of $X$ in the localized model structure $\caC_{\caS}$ (also called an $\caS$-local replacement). Similarly, if $\caC$ has a good $\caK$-colocalization, then a $\caK$-colocalization is a cofibrant replacement in the colocalized model structure $\caC^{\caK}$.

\begin{theorem}
Let $\caC$ be a cellular or combinatorial simplicial model category.
\begin{itemize}
\item[{\rm (i)}] If $\caC$ is left proper, 
then $\caC$ has a good $\caS$-localization for every set of maps $\caS$.
\item[{\rm (ii)}] If $\caC$ is right proper, 
then $\caC$ has a good $\mathcal{K}$-colocalization for every set of objects $\caK$. 
Moreover, the $\caK$-colocal objects is the smallest class of objects of $\caC$ that contains $\caK$ and is closed under homotopy colimits and weak equivalences.
\end{itemize}
\end{theorem}
\begin{proof}
For $\caC$ cellular see \cite[Theorem 4.1.1]{Hirschhorn} and \cite[Theorem 5.1.1, Theorem 5.1.5]{Hirschhorn}. 
If $\caC$ is combinatorial the result follows from~\cite[Theorem 4.7]{Barwick} and \cite[\S5]{Barwick}.
\end{proof}

If $F\colon \caC\to \caD$ is a left Quillen functor, denote by $\bL F\colon \Ho(\caC)\to \Ho(\caD)$ its \emph{left derived functor}. 
If $U\colon \caD\to \caC$ is a right Quillen functor, denote by $\bR U$ its \emph{right derived functor}.

\begin{lemma} 
\label{lemma:general}
Let $F\colon \caC \rightleftarrows \caD \colon U$ be a simplicial Quillen adjunction,
$\caS$ a set of homotopy classes of maps in $\caC$, and $\caK$ a set of isomorphism classes of objects of $\Ho(\caC)$. 
\begin{itemize}
\item[{\rm (i)}] An object $Z$ in $\caD$ is $\bL F(\caS)$-local if and only if $\bR U(Z)$ is $\caS$-local in $\caC$. 
Moreover, if $g$ is an $\caS$-local equivalence in $\caC$, then $\bL F(g)$ is an $\bL F(\caS)$-local equivalence in $\caD$.
\item[{\rm (ii)}] A map $f$ is an $\bL F(\caK)$-colocal equivalence in $\caD$ if and only if $\bR U(f)$ is a $\caK$-colocal equivalence in $\caC$. 
Moreover, if $W$ is $\caK$-colocal in $\caC$, then $\bL F(W)$ is $\bL F(\caK)$-colocal in $\caD$.
\end{itemize}
\end{lemma}
\begin{proof}
Both statements follow by using derived adjunctions.
\end{proof}

\section{Colored operads} 
\label{jnhgd}
In this appendix we recall the definitions and basic properties of colored operads and their algebras that are used in the paper. 
Throughout, $\mathcal{V}$ denotes a cocomplete closed symmetric monoidal category with tensor product $\otimes$, 
initial object $0$, unit $I$, and internal hom $\Hom_{\mathcal{V}}(-,-)$. 
The elements in the set $C$ are referred to as \emph{colors}.

\begin{definition}
A \emph{$C$\nobreakdash-colored collection} $\mathcal{K}$ in $\mathcal{V}$ consists of a set of objects $\mathcal{K}(c_1,\ldots, c_n;c)$ in $\mathcal{V}$ for each 
$(n+1)$-tuple of colors $(c_1,\ldots, c_n,c)$ equipped with a right action of the symmetric group $\Sigma_n$ given by maps
$$
\alpha^*\colon \mathcal{K}(c_1,\ldots, c_n;c)\longrightarrow
\mathcal{K}(c_{\alpha(1)},\ldots, c_{\alpha(n)};c),
$$
where $\alpha\in\Sigma_n$ (by default, $\Sigma_n$ is the trivial group if $n=0$ or $n=1$).
\end{definition}

A \emph{map} of $C$\nobreakdash-colored collections $\varphi\colon \mathcal{K}\longrightarrow\mathcal{L}$ consists of maps in $\mathcal{V}$
$$
\varphi_{c_1,\ldots,c_n;c}
\colon 
\mathcal{K}(c_1,\ldots,c_n;c)
\longrightarrow
\mathcal{L}(c_1,\ldots, c_n; c),
$$
for $(n+1)$-tuples $(c_1,\ldots, c_n,c)$, $n\ge 0$, that is compatible with the action of $\Sigma_n$. 
We denote by $\Coll_C(\mathcal{V})$ the category of $C$\nobreakdash-colored collections in $\mathcal{V}$.

\begin{definition}
A \emph{$C$\nobreakdash-colored operad} $\mathcal{O}$ in $\mathcal{V}$ is a $C$\nobreakdash-colored collection equipped with unit maps 
$I\longrightarrow \mathcal{O}(c;c)$ for every $c\in C$ and, 
for every $(n+1)$-tuple of colors $(c_1,\ldots, c_n,c)$ and $n$ given tuples
\[
(a_{1,1},\ldots, a_{1,k_1}; c_1),\ldots, (a_{n,1},\ldots, a_{n,k_n};
c_n),
\]
a \emph{composition product} map
$$
\xymatrix{
\mathcal{O}(c_1,\ldots, c_n;c)\otimes \mathcal{O}(a_{1,1},\ldots,
a_{1,k_1};c_1)\otimes\cdots\otimes \mathcal{O}(a_{n,1},\ldots,
a_{n,k_n};c_n)\ar[d]
\\ \mathcal{O}(a_{1,1},\ldots,a_{1,k_1},a_{2,1},\ldots,a_{2,k_2},\ldots,a_{n,1},\ldots,a_{n,k_n};c),
}
$$
that is compatible with the symmetric groups actions and subject to the associativity and unitary isomorphisms, see \cite[\S 2]{Elmendorf-Mandell}.
\end{definition}

A \emph{map of $C$\nobreakdash-colored operads} is a map of the underlying $C$\nobreakdash-colored collections that is compatible with the unit and composition product maps. 

Denote by $\mathcal{V}^C$ the product category of copies of $\mathcal{V}$ indexed by the set of colors $C$; that is, $\mathcal{V}^C=\prod_{c\in C}\mathcal{V}$.
For every object $\mathbf{X}=(X(c))_{c\in C}$ in $\mathcal{V}^C$, 
the \emph{endomorphism colored operad} $\End({\mathbf{X}})$ of $\mathbf{X}$ is the $C$-colored operad defined by
$$
\End(\mathbf{X})(c_1,\ldots, c_n; c)
:=
\Hom_{\mathcal{V}}(X(c_1)\otimes\cdots\otimes X(c_n),X(c)).
$$
Here, $X(c_1)\otimes\cdots\otimes X(c_n)$ is the unit $I$ when $n=0$. 
The composition product is ordinary composition and the $\Sigma_n$-action is defined by permutation of the factors.

\begin{definition}
Let $\mathcal{O}$ be any $C$-colored operad in $\mathcal{V}$. 
An \emph{$\mathcal{O}$-algebra} (or an \emph{algebra over $\mathcal{O}$}) $\cA$ is an object $\mathbf{X}=(X(c))_{c\in C}$ of $\mathcal{V}^C$ together with a map 
$ \mathcal{O}\longrightarrow \End(\mathbf{X}) $ of $C$-colored operads.
\end{definition}

Equivalently, 
since the monoidal category $\mathcal{V}$ is closed, 
an $\mathcal{O}$-algebra is a family of objects $X(c)$ in $\mathcal{V}$ for every $c\in C$ together with maps
$$
\mathcal{O}(c_1,\ldots, c_n; c)\otimes X(c_1)\otimes\cdots \otimes X(c_n)\longrightarrow X(c),
$$
for every $(n+1)$-tuple $(c_1,\ldots,c_n, c)$, that are compatible with the symmetric group action, the unit maps of $\mathcal{O}$, and subject to the usual associativity isomorphisms.

A \emph{map of $\mathcal{O}$-algebras} $\mathbf{f}\colon\cA\longrightarrow\cB$ is comprised of maps $(f_c\colon X(c)\longrightarrow Y(c))_{c\in C}$ of underlying collections inducing a
commutative diagram of $C$-colored collections
$$
\xymatrix{
\mathcal{O}\ar[r] \ar[d] & \End(\mathbf{X}) \ar[d] \\
\End(\mathbf{Y}) \ar[r]& \Hom(\mathbf{X},\mathbf{Y}).
}
$$
The top and left arrows are the given $\mathcal{O}$-algebra structures on $\mathbf{X}$ and $\mathbf{Y}$, respectively.
The $C$-colored collection $\Hom(\mathbf{X},\mathbf{Y})$ is defined as
$$
\Hom(\mathbf{X},\mathbf{Y})(c_1,\ldots, c_n;c)
:=\Hom_{\mathcal{V}}(X(c_1)\otimes\cdots\otimes X(c_n), Y(c)),
$$
and the right and bottom arrows are induced by the maps $f_c$.
If $\mathcal{V}$ has pullbacks, then a map $\mathbf{f}$ of $\mathcal{O}$-algebras can be viewed as a map of $C$-colored operads
$$
\mathcal{O}\longrightarrow \End(\mathbf{f}),
$$
where $\End(\mathbf{f})$ is the pullback of the diagram of $C$-colored collections
\begin{equation}
\xymatrix{
\End(\mathbf{f})\ar@{.>}[r] \ar@{.>}[d] & \End(\mathbf{X}) \ar[d] \\
\End(\mathbf{Y}) \ar[r]& \Hom(\mathbf{X},\mathbf{Y}).
}
\label{end_f}
\end{equation}
Note that $\End(\mathbf{f})$ inherits a $C$-colored operad structure from the $C$-colored operads $\End(\mathbf{X})$ and $\End(\mathbf{Y})$. 
We denote the category of $\mathcal{O}$-algebras by $\Alg_{\mathcal{O}}(\mathcal{V})$.

\begin{definition}
Given a $C$-colored operad $\mathcal{O}$ and an object $\mathbf{X}=(X(c))_{c\in C}$ in $\mathcal{V}^C$, the \emph{restricted endomorphism operad} $\End_{\mathcal{O}}(\mathbf{X})$ 
is defined by
\begin{equation}
\End_{\mathcal{O}}(\mathbf{X})(c_1,\ldots, c_n; c)
:=
\left\{
\begin{array}{ll}
\End(\mathbf{X},\mathbf{Y})(c_1,\ldots, c_n;c) & \mbox{if }\mathcal{O}(c_1,\ldots, c_n;c)\ne 0, \\[0.2cm]
0 & \mbox{otherwise}.
\end{array}
\right.
\label{endp}
\end{equation}
\end{definition}
There is a canonical inclusion of $C$-colored operads $\End_{\mathcal{O}} (\mathbf{X})\longrightarrow \End(\mathbf{X})$, 
and thus every map $\mathcal{O}\longrightarrow \End(\mathbf{X})$ of $C$-colored operad factors uniquely through the restricted endomorphism operad $\End_{\mathcal{O}}(\mathbf{X})$.
Hence an $\mathcal{O}$-algebra structure on $\mathbf{X}$ is given by a map of $C$-colored operads $\mathcal{O}\longrightarrow \End_{\mathcal{O}}(\mathbf{X})$.

If $\alpha\colon C\longrightarrow D$ is a function between sets of colors,
any $D$-colored operad $\mathcal{O}$ pulls back to a $C$-colored operad $\alpha^*\mathcal{O}$ and there is an adjoint functor pair (see~\cite[\S 1.6]{BM07})

\begin{equation}
\label{coloradjoint1}
\xymatrix{
\alpha_{!} : {\Oper}_C(\mathcal{V})\ar@<3pt>[r] & \ar@<3pt>[l] {\Oper}_D(\mathcal{V}) : \alpha^*.
}
\end{equation}
The restriction functor $\alpha^*$ is defined by 
$
(\alpha^* \mathcal{O})(c_1,\ldots,c_n;c)
:=\mathcal{O}(\alpha(c_1),\ldots, \alpha(c_n); \alpha(c))$.
A function $\alpha\colon C\longrightarrow D$ also defines an adjoint pair between the corresponding categories of algebras for every $D$-colored operad 
$\mathcal{O}$ in $\mathcal{V}$, 
i.e., 
\begin{equation}
\label{coloradjoint2}
\xymatrix{ 
\alpha_{!} : {\Alg}_{\alpha^*\mathcal{O}}(\mathcal{V})\ar@<3pt>[r] & \ar@<3pt>[l] {\Alg}_{\mathcal{O}}(\mathcal{V}) : \alpha^*.
}
\end{equation}
If $\cA$ is an $\mathcal{O}$\nobreakdash-algebra with structure map $\gamma\colon \mathcal{O}\longrightarrow \End({\bf X})$, 
then $(\alpha^*{\bf X})(c):=X(\alpha(c))$ for all $c\in C$, with structure map defined by (\ref{coloradjoint1}),
i.e., 
\begin{equation}
\label{alphastar}
\alpha^*\gamma\colon \alpha^*\mathcal{O}
\longrightarrow 
\alpha^*\End({\bf X})=\End(\alpha^*{\bf X}).
\end{equation}

When $C=\{c\}$, 
a $C$-colored operad $\mathcal{O}$ is an operad, where $\mathcal{O}(n)$ is short for $\mathcal{O}(c,\ldots,c;c)$ with $n\ge 0$ inputs. 
The \emph{associative operad} $\Ass$ is the one-color operad with $\Ass(n)=I[\Sigma_n]$ for $n\ge 0$.
Here, $I[\Sigma_n]$ is the coproduct of copies of the unit $I$ indexed by $\Sigma_n$, on which $\Sigma_n$ acts freely by permutations. 
The \emph{commutative operad} $\Com$ is the one-color operad with $\Com(n)=I$ for $n\ge 0$. 
Algebras over $\Ass$ are the associative monoids in $\mathcal{V}$, while algebras over $\Com$ are the commutative monoids in $\mathcal{V}$.

For $\mathcal{O}$ a one\nobreakdash-colored operad in $\mathcal{V}$,
let $\modu_{\mathcal{O}}$ be the $C$-colored operad with colors $C=\{r,m\}$ and nonzero terms $\modu_{\mathcal{O}}(r,\stackrel{(n)}{\ldots}, r;r):=\mathcal{O}(n)$, $n\ge 0$, and 
$\modu_{\mathcal{O}}(c_1,\ldots, c_n;m):=\mathcal{O}(n)$, $n\ge 1$, where exactly one $c_i$ is $m$ and the rest (if any) are equal to $r$. 
An algebra over $\modu_{\mathcal{O}}$ is a pair $(\cR, \cM)$ of objects of $\mathcal{V}$, 
where $\cR$ is an $\mathcal{O}$\nobreakdash-algebra and $\cM$ is a module over~$\cR$.
That is,  
an object equipped with maps
\[
\mathcal{O}(n) \otimes \cR \otimes {\stackrel{(k-1)}{\cdots}} \otimes \cR \otimes \cM \otimes \cR \otimes {\stackrel{(n-k)}{\cdots}}
\otimes \cR \longrightarrow \cM
\]
for $n\ge 1$ and $1\le k\le n$, that are equivariant and compatible with associativity isomorphisms and the unit of $\mathcal{O}$.

When $\mathcal{O}=\Ass$, an algebra over $\modu_{\mathcal{O}}$ is a pair $(\cR,\cM)$ where $\cR$ is a monoid in~$\mathcal{V}$ and $\cM$ is an $\cR$\nobreakdash-bimodule, 
i.e., an object equipped with commuting left and right $\cR$\nobreakdash-actions. 
When $\mathcal{O}=\Com$, then $\cR$ is a commutative monoid in $\mathcal{V}$ and $\cM$ is a module over it (indistinctly left or right).

Let $\alpha$ denote the inclusion of $\{r\}$ into $\{r,m\}$.
Then $\alpha^*\modu_{\mathcal{O}} = \mathcal{O}$ for every operad $\mathcal{O}$, and $\alpha^*(\cR,\cM)=\cR$ for the corresponding algebras.

{\bf Acknowledgments.}
The authors gratefully acknowledge support by the RCN grant no.~239015 "Special Geometries",
RCN Frontier Research Group Project no.~250399 "Motivic Hopf Equations'', 
and the DFG Priority Program 1786 "Homotopy theory and algebraic geometry''.
Guti\'errez is supported by the Spanish Ministry of Economy under the grants MTM2016-76453-C2-2-P (AEI/FEDER, UE) and RYC-2014-15328 (Ram\'on y Cajal Program). {\O}stv{\ae}r is supported by a Friedrich Wilhelm Bessel Research Award from the Humboldt Foundation and a Nelder Visiting Fellowship from Imperial College London.

%\bibliographystyle{plain}
%\small{
%\bibliography{loccolocalg}
%}
\def\cprime{$'$}

\vspace{0.1in}

\begin{center}
Departament de Matem\`atiques i Inform\`atica, Universitat de Barcelona, Spain.\\
e-mail: javier.gutierrez@ub.edu
\end{center}
\begin{center}
Mathematisches Institut, Universit\"at Osnabr\"uck, Germany.\\
e-mail: oroendig@uni-osnabrueck.de
\end{center}
\begin{center}
Mathematisches Institut, Universit\"at Osnabr\"uck, Germany.\\
e-mail: markus.spitzweck@uni-osnabrueck.de
\end{center}
\begin{center}
Department of Mathematics, University of Oslo, Norway.\\
e-mail: paularne@math.uio.no
\end{center}
\end{document}